\documentclass[12pt,a4paper]{amsart}
\usepackage{graphicx}
\usepackage{color}
\usepackage{subfigure}
\usepackage{amsmath, amssymb, amsthm, amsfonts}

\allowdisplaybreaks
\usepackage{amsmath,amssymb,amsthm,amsfonts,mathrsfs,enumerate,sansmath,mathtools,amscd
}
\usepackage[colorlinks=true]{hyperref}
\hypersetup{urlcolor=blue, citecolor=blue, draft=false}

\newtheorem{theorem}{Theorem}[section]

\newtheorem{lemma}[theorem]{Lemma}
\newtheorem{corollary}[theorem]{Corollary}
\theoremstyle{definition}
\newtheorem{definition}[theorem]{Definition}

\newtheorem{remark}[theorem]{Remark}

\begin{document}
\title[Continuity and bi-Lipschitz properties] 
{Continuity and bi-Lipschitz properties of the Hurwitz and its invariant metrics}

\author[Arstu and S. K. Sahoo]{Arstu$^\dagger$ and Swadesh Kumar Sahoo$^\dagger$}

\address{$^\dagger$Department of Mathematics, Indian Institute of Technology Indore, 
Indore--453 552, India}
\email{arstumothsra@gmail.com}
\email{swadesh.sahoo@iiti.ac.in}




\begin{abstract}
This paper attempts to study the continuity of the Hurwitz metric in arbitrary proper subdomains of the complex plane
and to introduce a new invariant metric bi-Lipschitz 
equivalent to the Hurwitz metric in hyperbolic domains. The lower semi-continuity and other basic properties of this invariant metric are also presented. 
\end{abstract}

\maketitle


\section{Introduction }\label{sec1}
\setcounter{equation}{0}


The classical Poincar\'e's hyperbolic metric was first introduced in the early nineties. Since then, 
a family of conformal metrics that are closely related to the hyperbolic metric 
was introduced by several mathematicians. Just to name a few, they are 
the Hurwitz metric \cite{Min16}, the Gardiner-Lakic metric \cite{GL01},
the Hahn metric \cite{Hahn81}, the quasihyperbolic metric \cite{GP76},
and many more. 
As the hyperbolic metric is valid for only hyperbolic domains of the plane due 
to its very difficult nature to compute explicitly, these metrics play a vital 
role to enhance the study of the hyperbolic metric in several purposes.
One such important problem is bi-Lipschitz equivalence of the hyperbolic metric
with the other conformal metrics. For instance, the Gardiner-Lakic metric in a hyperbolic domain is 
defined by taking the supremum of the hyperbolic metric in twice punctured plane, punctured at the 
distinct pair of points in the complement of the domain and the supremum is taken over these pair of points. 
Note that, similar to the hyperbolic metric and other related metrics, 
the Gardiner-Lakic metric satisfies the domain monotonicity property. 
Bi-Lipschitz equivalence of the hyperbolic and the Gardiner-Lakic metrics is proved in \cite{GL01}. 
Some other properties including the continuity, M\"obius invariance of the Gardiner-Lakic metric 
and other related metrics with the hyperbolic metric in hyperbolic domains are studied by 
Herron et al. in 2008 (see \cite{HMM08}). 

In this research, we adopt the definition of the Gardiner-Lakic 
metric of the hyperbolic metric to define a new metric by replacing the hyperbolic metric with the 
Hurwitz metric.
We show that this is invariant under M\"obius 
transformations which fix $\infty$ in hyperbolic domain of the Complex plane and M\"obius invariant in hyperbolic domain of the Riemann sphere. Continuity of the Hurwitz metric play a vital role for proving Lower semi-continuity of  the Hurwitz metric in the sense of Gardiner and Lakic in hyperbolic domains. Furthermore, 
we do establish a bi-Lipschitz equivalence of this new metric with the Hurwitz metric, the hyperbolic
and the quasihyperbolic metrics.
However, on hyperbolic domain with connected boundary, we give a sharper bi-Lipschitz constant for the first case.

\section{Preliminaries} 
In this work, relatively standard notations are used. First, the complex plane is denoted by $\mathbb{C}$
and the unit disk $\{z\in\mathbb{C}:\,|z|<1\}$ is denoted by $\mathbb{D}$ and the punctured unit disk $\mathbb{D}\setminus\{0\}$ takes the notation $\mathbb{D}^*$.. 
By a planar domain, we mean an open and connected subset of $\mathbb{C}$. Unless specified, 
$\Omega$ denotes a hyperbolic domain in $\mathbb{C}$, i.e. its complement 
$\Omega^c:=\mathbb{C}\setminus \Omega$ possesses at least two points. 
Each such $\Omega$ carries a unique maximal constant curvature $-1$ associated with the conformal metric 
$$
\lambda_\Omega(w)\,|dw|=\lambda_{\mathbb{D}}(z)\,|dz|=\frac{2}{1-|z|^2}\,|dz|
$$
referred as the {\em Poincar\'e hyperbolic metric} in $\Omega$, where $w=f(z)$ is a universal
covering map of $\mathbb{D}$ onto $\Omega$. This is independent of the choice of
$f$ in the sense that the above relation continues to hold if $f$ is replaced by 
the composition $f\circ h$, where $h:\,\mathbb{D}\to \mathbb{D}$ is a M\"obius transformation.

In association with the hyperbolic metric,
Gardiner-Lakic \cite{GL01} introduced a M\"obius 
invariant metric $\kappa_\Omega(w)\,|dw|$ which is defined by 
$$
\kappa_\Omega(w)=\sup_{a,b\in \Omega^c}\lambda_{\mathbb{C}\setminus\{a,b\}}(w)
$$
for $w\in\Omega$ and distinct points
$a,b$. Further, Gardiner and Lakic proved in the same paper (see \cite[Theorem~3]{GL01}) 
that the $\kappa_\Omega$-metric is bi-Lipschitz equivalent to the hyperbolic metric 
$\lambda_\Omega$. However, discussions on various improvements in the upper bi-Lipschitz 
constant are taken places later in
\cite{HMM08,SV05}.

By the notation $\partial\Omega$, we mean the boundary of $\Omega$.
The Hurwitz metric $\eta_\Omega(w)\,|dw|$ is defined by 
$$\eta_\Omega(w)=\frac{2}{r_\Omega(w)}, \quad r_\Omega(w)=\max\{f'(0):\, f\in\mathcal{H}_0^w(D,\Omega)\},
$$
where $\mathcal{H}_0^w(D,\Omega)$
stands for the collection of all holomorphic functions $f$ from $\mathbb{D}$ to $\Omega$ such 
that $f(0)=w$ and $f(s)\neq w$ for all $s\in\mathbb{D}^*$ and $ f'(0)> 0$. 
The collection of all holomorphic functions from $\mathbb{D}$ to $\Omega$ is
denoted by $\mathcal{H}(\mathbb{D},\Omega)$.
The existence of the extremal function for the Hurwitz metric is established by Minda in 
\cite[Theorem~4.1]{Min16}. 
Moreover, he proved that this extremal function is unique and called as the {\em Hurwitz covering map}. 
Note that the hyperbolic metric is always dominated by the Hurwitz metric.

According to \cite{LM12}, for any two sets $X$ and $Y$, the Hausdorff distance $H(X,Y)$ between them is defined as 
$$
H(X,Y):=\inf\{r>0:X\subset N_r(Y)\mbox{~and~} Y\subset N_r(X)\},
$$
where $N_r(X)$ denotes the collection of all points of $X$ whose Euclidean distance from $X$ is less than $r$. 
An equivalent definition of the Hausdorff distance is as follows:
$$
H(X,Y):=\max\big(\sup_{w\in X}d(w,Y),\sup_{w\in Y}d(w,X)\big),
$$
where $d(w,X):=\inf_{z\in X}\{|z-w|\}$.
Different types of convergence of a sequence of sets is studied in literature, for example the Carath\'eodory kernel convergence and 
convergence in boundary. For the sake of necessity in our results, the convergence in boundary is defined as:

\begin{definition}\cite{LM12}\label{def2.1}
Let $\Omega_n$ be a sequence of domains and $\Omega\subset\mathbb{C}$ be a domain. 
Then $\Omega_n$ is said to converge in boundary to $\Omega$ if
\begin{enumerate}
\item $H(\partial\Omega,\partial\Omega_n)\to 0$ as $n\to \infty$; and
\item there exits $w_0\in\Omega$ such that $w_0\in\Omega_n$ for all but finitely many $n\in\mathbb{N}$.
\end{enumerate} 
\end{definition}


\section{Continuity of the Hurwitz metric}\label{p3sec3}

Recently the continuity of the Hurwitz metric
in bounded planar domains is established by Sarkar and Verma (see \cite[Theorem~1.4]{SV20}) and in the same paper 
they made curvature calculation as well for arbitrary planar domains.
The idea of scaling principle for planar domains is adopted for these studies from \cite{GKK11}.
In this note, we provide an alternative and shorter proof of the continuity of the Hurwitz metric in 
arbitrary planar domains. The concept of the so-called convergence of sequence of domains in boundary is used 
in our investigation as a main tool (see the following lemma). However, this depends on the Hausdorff distance
(see Definition~\ref{def2.1}).

\begin{lemma}\label{lem1}
Let $\Omega$ be a planar domain. If $w_n\in\Omega$ is a sequence of points converging to a point 
$w\in\Omega$, then the corresponding sequence of punctured domains $\Omega_{w_n}=\Omega\setminus\{w_n\}$ 
converges in boundary to 
$\Omega_w=\Omega\setminus\{w\}$.  
\end{lemma}
\begin{proof}
To prove our lemma, we first demonstrate that the Hausdorff distance 
$H(\partial\Omega_w,\partial\Omega_{w_n})$ approaches zero as $n\to \infty$. Consider 
\begin{align*}
H(\partial\Omega_w,\partial\Omega_{w_n})
&=
\max\big(\sup_{w\in\partial\Omega_w}
d(w,\partial\Omega_{w_n}),\sup_{w\in \partial\Omega_{w_n}}d(w,\partial\Omega_w)\big)\\
&=
\max\big(|w-w_n|,|w_n-w|\big)\\
&= 
|w_n-w|,
\end{align*}
where the second equality follows from the simple observation that 
$\partial\Omega_{w_n}=\partial\Omega\cup\{w_n\}$ and $\partial\Omega_{w}=\partial\Omega\cup\{w\}$. 
Note that $w_n$ converges to $w$. Thus we conclude that the Hausdorff distance 
$H(\partial\Omega_w,\partial\Omega_{w_n})\to 0$ as $n\to\infty$.

Since $\Omega$ is a domain and $w\in\Omega$, we can choose a small ball $B(w,r)\subset\Omega$ and 
hence a point $w_0\in B(w,r)$ such that $w_0\neq w_n,w$ for $n\in\mathbb{N}$. 
As a consequence, $w_0$ belongs to $\Omega_w$ as well as $\Omega_{w_n}$. Hence, by the definition, we conclude 
that $\Omega_{w_n}$ converges in boundary to $\Omega_w$.  
\end{proof}


It is well known in literature that the classical hyperbolic metric \cite{ke07}, the 
Hahn metric \cite{Min83}, the capacity metric \cite{Sui73} are continuous functions. 
In this direction, as an application of Lemma~\ref{lem1},
we now show that the Hurwitz metric is also a continuous function. 

\begin{theorem}\label{thm1}
Let $\Omega\subsetneq\mathbb{C}$ be a domain, then the Hurwitz density $\eta_{\Omega}$ 
is a continuous function.
\end{theorem}
\begin{proof}
Let $w\in\Omega$ be an arbitrary point and $w_n$ be a sequence of points in $\Omega$ converging to $w$. Then we show that $\eta_{\Omega}(w_n)\to\eta_{\Omega}(w)$ as $n\to\infty$. 
By \cite[Theorem~7.1]{Min16}, we have
$$
|z-w_n|\log\left(\cfrac{2}{\eta_{\Omega}(w_n)|z-w_n|}\right)\lambda_{\Omega_{w_n}}(z)
=1+O\left(\cfrac{|z-w_n|}{\log\left(\cfrac{2}{\eta_{\Omega}(w_n)|z-w_n|}\right)}\right).
$$
To solve the above equation for $\eta_{\Omega_n}$, we do the following calculations$:$
\begin{align}
\cfrac{1}{|z-w_n|\lambda_{\Omega_{w_n}}(z)}
&=\cfrac{\log\left(\cfrac{2}{\eta_{\Omega}(w_n)|z-w_n|}\right)}{1+O\left(\cfrac{|z-w_n|}{\log\left(\cfrac{2}{\eta_{\Omega}(w_n)|z-w_n|}\right)}\right)}\nonumber\\
&=\log\left(\cfrac{2}{\eta_{\Omega}(w_n)|z-w_n|}\right)\left(1-O\left(\cfrac{|z-w_n|}{\log\left(\cfrac{2}{\eta_{\Omega}(w_n)|z-w_n|}\right)}\right)\right)\nonumber\\
&=\log\left(\cfrac{2}{\eta_{\Omega}(w_n)}\right)-\log(|z-w_n|)-O(|z-w_n|)\nonumber.
\end{align}
This implies that
$$
\log\left(\cfrac{\eta_{\Omega}(w_n)}{2}\right)
=\cfrac{1}{\log(|z-w_n|)}-\cfrac{1}{|z-w_n|\lambda_{\Omega_{w_n}}(z)}-O(|z-w_n|).
$$

Let $f(z)=O(|z-w_n|)$ as $z\to w_n$. By definition of big $O$, there exist $M,\epsilon>0$ such that for all $z$ with $0<|z-w_n|<\epsilon$ we have $|f(z)|\le M|z-w_n|$. Since $w_n\to w$, for $\delta>0$ there exists a number $N\in\mathbb{N}$ such that $|w_n-w|<\delta$ for all $n\ge N$. Using the triangle inequality $|z-w|\le|z-w_n|+|w_n-w|$ and choosing $\delta=\epsilon$ we have $|f(z)|\le M|z-w|$ for all $z$ with $0<|z-w|<2\epsilon$. Therefore, we conclude that $O(|z-w_n|)\to O(|z-w|)$ as $n\to \infty$.

Recall that the logarithm function and $1/|z-w_n|$ are continuous. In view of \cite[Lemma~2]{LM12},  $\lambda_{\Omega_{w_n}}$ 
converges locally uniformly to $\lambda_{\Omega_w}$ whenever the sequence of domains $\Omega_{w_n}$ 
converges in boundary to $\Omega_w$. These facts along with Lemma \ref{lem1} give us 
$$
\lim_{n\to\infty}\log\cfrac{\eta_{\Omega}(w_n)}{2}
=\log\cfrac{1}{|z-w|}-\cfrac{1}{|z-w|\lambda_{\Omega_w}(z)}-O(|z-w|).
$$
Therefore, the continuity of the Hurwitz metric follows.
\end{proof}
 
 
\section{An invariant metric}
Minda introduced the Hurwitz metric and studied the bi-Lipschitz equivalence of the Hurwitz metric with the quasihyperbolic metric (see \cite{Min16}). Note that the quasihyperbolic metric is not M\"obius invariant. A natural question arises:{\em  does there exist a M\"obius invariant metric which is bi-Lipschitz equivalent to the Hurwitz metric?} This question motivates us to define a new metric which is obtained by refining the Hurwitz metric in the sense of the Gardiner-Lakic metric. Indeed, we show that this new metric is invariant under M\"obius maps that fix $\infty$ and is bi-Lipschitz equivalent to the Hurwitz metric. As a consequence, on hyperbolic domain with uniformly perfect boundary, the Gardiner-Lakic version of the Hurwitz metric and the classical hyperbolic metric are also bi-Lipschitz equivalent. Some other important basic properties of this new metric are also studied in this section.  
 
Let $\Omega\subsetneq\mathbb{C}$ be a domain. For $w\in\Omega$, the {\em quasihyperbolic density} is defined by $1/\delta_\Omega(w)$, where $\delta_\Omega(w)=\inf_{p\in\Omega^c}|w-p|$, the distance of $w$ to the boundary of $\Omega$. Note that the Hurwitz metric is bi-Lipschitz equivalent to the quasihyperbolic metric on proper domains of the complex plane. It is quite interesting to study the quasihyperbolic metric in the Gardiner-Lakic sense, which we define as follows: 
$$
\cfrac{1}{\overline{\delta}_\Omega(w)}=\sup\cfrac{1}{\delta_{\mathbb{C}\setminus\{w_1,w_2\}}(w)},
$$
where the supremum is taken over all {\em distinct} pair of points $w_1,w_2\in\Omega^c.$
However, surprisingly, we demonstrate here that $1/\overline{\delta_\Omega}$ and $1/\delta_\Omega$ agree on the hyperbolic domains. Indeed, we have
$$
\cfrac{1}{\overline{\delta}_\Omega(w)}
=\sup_{w_1,w_2\in \Omega^c}\cfrac{1}{\delta_{\mathbb{C}\setminus\{w_1,w_2\}}(w)}
=\sup_{w_1,w_2\in \Omega^c}\left\{\cfrac{1}{|w-w_1|},\cfrac{1}{|w-w_2|}\right\}
=\sup_{w_1\in \Omega^c}\cfrac{1}{|w-w_1|}=\cfrac{1}{\delta_\Omega}.
$$
This justifies the introduction of the Gardiner-Lakic version of the Hurwitz metric in the sequel instead  of the quasihyperbolic metric. 

For a hyperbolic domain $\Omega$ and $w\in\Omega$, setting 
$$
\overline{\eta}_\Omega(w)=\sup\eta_{\mathbb{C}\setminus\{w_1,w_2\}}(w),
$$
where the supremum is taken over all {\em distinct} pair of points $w_1,w_2\in\Omega^c.$
Unfortunately, our new metric $\overline{\eta}_\Omega$ is also difficult to compute in view of the nature of the Hurwitz metric in twice punctured planes. 

\begin{remark}
For a hyperbolic domain $\Omega$ and $w_1,w_2\in\Omega^c$, it is easy to see that 
$\Omega$ is contained in $\mathbb{C}\setminus\{w_1,w_2\}$. By the domain monotonicity property of 
the Hurwitz metric, it follows that $\eta_{\mathbb{C}\setminus\{w_1,w_2\}}(w)\leq\eta_{\Omega}(w)$ 
for all $w\in\Omega$ and hence on taking supremum over all $w_1,w_2\in\Omega$, we 
have $\overline{\eta}_\Omega(w)\leq\eta_{\Omega}(w)$ for all $w\in\Omega$.
\end{remark}
\begin{remark}
Suppose $\Omega_1\subset\Omega_2$ are hyperbolic domains, then $\Omega_2^c\subset\Omega_1^c$. 
Therefore, by the definition of $\overline{\eta}$, we have $\overline{\eta}_{\Omega_1}\geq\overline{\eta}_{\Omega_2}$. 
That is, $\overline{\eta}$ also satisfies the {\em domain monotonicity property}.  
\end{remark}

In the definition of $\overline{\eta}_\Omega$, the supremum is always attained for a pair of points 
in the complement of $\Omega$. We now illustrate this in the form of the following lemma.

\begin{lemma}\label{p3lem1}
Let $\Omega$ be a hyperbolic domain. For every $w\in\Omega$ there exist distinct points $p,q\in\Omega^c$ 
for which supremum is attained for $\overline{\eta}_\Omega(w)$, that is $\overline{\eta}_\Omega(w)
=\eta_{\mathbb{C}\setminus\{p,q\}}(w)$. 
\end{lemma}
\begin{proof}
Fix $w_0\in\Omega$. Let $p_n\neq q_n\in\Omega^c$ be sequences such that $\eta_{\mathbb{C}
\setminus\{p_n,q_n\}}(w_0)\to\overline{\eta}_\Omega(w_0)$. This is possible by the definition of supremum. Since $\Omega^c$ is a closed set, 
limit of the sequences $p_n,q_n$ say $p$ and $q$ belong to $\Omega^c$ itself. For $w_0\in\Omega$, 
let $g_n$ and $g$ be the Hurwitz covering maps from $\mathbb{D}\to\mathbb{C}\setminus\{p_n,q_n\}$ 
and $\mathbb{D}\to\mathbb{C}\setminus\{p,q\}$, respectively. 

Note that the conformal map $T_n(w)=[(w-p)(q_n-p_n)/(q-p)]+p_n$ maps
$\mathbb{C}\setminus\{p,q\}$ onto $\mathbb{C}\setminus\{p_n,q_n\}$. It is easy to see that 
the composition function $T_n\circ g$ from $\mathbb{D}$ onto $\mathbb{C}\setminus\{p_n,q_n\}$ 
is a Hurwitz covering map. By the uniqueness of the Hurwitz covering map, we compute
\begin{equation}\label{p3eq3}
g_n(w_0)=(T_n\circ g)(w_0)=\left(\cfrac{g(w_0)-p}{q-p}\right)(q_n-p_n)+p_n.
\end{equation}
Taking $n\to\infty$ in \eqref{p3eq3}, the limit $g_n\to g$ follows. Also, the restricted holomorphic covering map $g_n:\mathbb{D}\setminus\{0\}\to\mathbb{C}\setminus\{p_n,q_n,w_0\}$ 
converges $g:\mathbb{D}\setminus\{0\}\to\mathbb{C}\setminus\{p,q,w_0\}$. 
Hence the points $p,q$ are distinct; otherwise it converges locally uniformly to $w_0$ times the rotation map of $\mathbb{D}\setminus\{0\}$ \cite{Hej74}. Now, the result follows from the uniqueness of the limit.
\end{proof}

It is always interesting to know whether a metric is M\"obius invariant. In literature many metrics 
such as the hyperbolic metric \cite{ke07}, the Hurwitz metric \cite{Min16}, the Kobayashi density of the 
Hurwitz metric \cite{AS1}, the Carath\'eodory density of the Hurwitz metric \cite{AS2} and the Gardiner-Lakic 
metric \cite{GL01} are M\"obius invariant. Precisely, M\"obius invariance of the Kobayashi density of the Hurwitz metric follows by applying a M\"obius map and its inverse in \cite[Theorem~3.11]{AS1}. We now demonstrate that $\overline{\eta}_\Omega$ is 
M\"obius invariant when it fixes $\infty$.

Analogue to the definition of the hyperbolic metric in hyperbolic domains of the Riemann sphere, it is easy to extend the notion of the Hurwitz metric on any proper subdomain of the Riemann sphere whose complement contains at least two points. As a consequence, the $\bar{\eta}$ metric can be extended to the hyperbolic domains in the Riemann sphere. We now demonstrate that $\overline{\eta}_\Omega$ is invariant under M\"obius maps which fixes the point infinity (affine maps). However, following the steps of Lemma~\ref{p3lem2}, M\"obius Invariance of $\overline{\eta}_\Omega$ can be guaranteed whenever $\Omega$ is a hyperbolic domain in Riemann sphere. 	

\begin{lemma}\label{p3lem2}
The metric $\overline{\eta}_\Omega$ is invariant under M\"obius transformations which fix infinity. 
\end{lemma}
\begin{proof}
Suppose that $T$ is a M\"obius transformation 
and $T(\Omega_1)=\Omega_2$, where $\Omega_1$ is a hyperbolic domain. In the definition of $\overline{\eta}_\Omega$,
since $w_1,w_2\in\Omega_1^c$, 
their images under $T$ belong to $\Omega_2^c$, that is $T(w_1)=s_1,T(w_2)=s_2\in\Omega_2^c$ 
and $s_1\neq s_2$. In order to prove that $\overline{\eta}_\Omega$ is M\"obius invariant, it is enough to show that
 $$
\overline{\eta}_{\Omega_2}(T(w))|T'(w)|=\overline{\eta}_{\Omega_1}(w). 
 $$
 As a consequence of distance decreasing property of the Hurwitz density we have
 $$
 \eta_{\mathbb{C}\setminus\{s_1,s_2\}}(T(w))|T'(w)|=\eta_{\mathbb{C}\setminus\{w_1,w_2\}}(w). 
 $$
 By the definition of $\overline{\eta}_{\Omega_2}$, it follows that 
 $$
 \overline{\eta}_{\Omega_2}(T(w))|T'(w)|\geq\eta_{\mathbb{C}\setminus\{w_1,w_2\}}(w). 
 $$
 Therefore, we have $\overline{\eta}_{\Omega_2}(T(w))|T'(w)|\geq\overline{\eta}_{\Omega_1}(w)$. 
 The reverse inequality is guaranteed by applying the same argument to $T^{-1}$.
 \end{proof}
    

In Theorem~\ref{thm1}, we proved the continuity of the Hurwitz metric. Our next result 
demonstrates the lower semi-continuity of the $\overline{\eta}_\Omega$ metric by 
using the fact that the Hurwitz metric is continuous. We are looking forward to establish the continuity of $\overline{\eta}_\Omega$, however, at present we do not have a proof for it.

\begin{theorem}
For a hyperbolic domain $\Omega$, the density function $\overline{\eta}_\Omega$ is a lower 
semi-continuous function.
\end{theorem}
\begin{proof}
Let $w_0\in\Omega$ be a fixed point. To prove the lower semi-continuity, we shall demonstrate that 
$$
\overline{\eta}_\Omega(w_0)\leq\liminf_{w\to w_0}\overline{\eta}_\Omega(w).
$$
Choose a sequence $w_n\in\Omega$ converging to the point $w_0$ such that 
$\overline{\eta}_\Omega(w_n)\to\liminf_{w\to w_0}\overline{\eta}_\Omega(w)$. 
From Lemma \ref{p3lem1}, there exit points $p,q\in\Omega^c$ with $\overline{\eta}_\Omega(w_0)
=\eta_{\mathbb{C}\setminus\{p,q\}}(w_0)$. By the definition of $\overline{\eta}_\Omega$ 
it follows that $\eta_{\mathbb{C}\setminus\{p,q\}}(w_n)\leq\overline{\eta}_\Omega(w_n)$ 
for all $n\in\mathbb{N}$. Thus, we have
$$
\overline{\eta}_\Omega(w_0)
=\eta_{\mathbb{C}\setminus\{p,q\}}(w_0)
=\lim_{n\to\infty}\eta_{\mathbb{C}\setminus\{p,q\}}(w_n)
\leq\liminf_{n\to\infty}\overline{\eta}_\Omega(w_n)
\leq\liminf_{w\to w_0}\overline{\eta}_\Omega(w),
$$
where the second equality follows from Theorem~\ref{thm1}.
\end{proof}


\section{Bi-Lipschitz properties}
Minda proved in \cite[Theorem~6.4]{Min16} that the Hurwitz metric and the quasihyperbolic metric are bi-Lipschitz in any proper subdomain of the complex plane. The following result provides a bi-Lipschitz equivalence of the $\overline{\eta}_\Omega$ and the quasihyperbolic densities in a 
hyperbolic domain $\Omega$. 

\begin{theorem}\label{p3eq2}
Let $\Omega$ be a hyperbolic domain, then the quasihyperbolic density $\delta_\Omega$ and the 
density $\overline{\eta}_\Omega$ are bi-Lipschitz equivalent, that is 
$$
\cfrac{1}{8\delta_\Omega(w)}\leq\overline{\eta}_\Omega(w)\leq\cfrac{2}{\delta_\Omega(w)},
$$
for every  $w\in\Omega$. 
\end{theorem}
\begin{proof}
Let $p\in\partial\Omega$ be the nearest point of $w$, that is 
$|w-p|=\delta_\Omega(w)$. Since $\Omega$ is a hyperbolic domain, we can find a point 
$q\in\Omega^c$ with $q\neq p$. It is easy to see that $(\mathbb{C}\setminus\{p,q\})
\subset(\mathbb{C}\setminus\{p\})$. By the domain monotonicity property of the Hurwitz density, we have
\begin{equation}\label{p3eq1}
\eta_{\mathbb{C}\setminus\{p,q\}}(w)\geq\eta_{\mathbb{C}\setminus\{p\}}(w)
\end{equation}
for every $w\in\Omega$. Taking the supremum over $p,q\in\Omega^c$ in \eqref{p3eq1} and 
using the fact that $\eta_{\mathbb{C}\setminus\{p\}}(w)=1/(8|w-p|)$, we obtain
\begin{align*}
\overline{\eta}_\Omega(w)
=\sup_{w_1,w_2\in \Omega^c}\eta_{\mathbb{C}\setminus\{w_1,w_2\}}(w)
\geq\sup_{w_1\in \Omega^c}\eta_{\mathbb{C}\setminus\{w_1\}}(w)
&=\sup_{w_1\in \Omega^c}\cfrac{1}{8|w-w_1|}\\
&=\cfrac{1}{\inf_{w_1\in\Omega^c}8|w-w_1|}
=\cfrac{1}{8\delta_\Omega(w)}.
\end{align*}
The left inequality follows.

For the right inequality, we consider the open ball 
$$
B(w,\delta_\Omega(w))=\{z\in\Omega : |w-z|<\delta_\Omega(w)\}
\subset\mathbb{C}\setminus\{a,b\}
$$ 
for every $a,b\in\Omega^c$. Then it is easy to see by the domain 
monotonicity property of the Hurwitz density that for $w\in\Omega$
\begin{equation}
\overline{\eta}_\Omega(w)\leq\eta_{B(w,\delta(w))}(w)=\cfrac{2}{\delta_\Omega(w)},
\end{equation}
where the equality follows bt the fact that the Hurwitz and the hyperbolic densities coincide on 
simply connected domains. 
\end{proof}


In the next two corollaries, the bi-Lipschitz equivalence of the density $\bar\eta_\Omega$ with the 
Hurwitz density followed by with the hyperbolic density are provided.

\begin{corollary}\label{p3cor1}
The Hurwitz density $\eta_{\Omega}$ and the density $\overline{\eta}_\Omega$ are bi-Lipschitz 
equivalent in the hyperbolic domain $\Omega$. Precisely, we have
$$
\cfrac{\eta_\Omega(w)}{16}\leq\overline{\eta}_\Omega(w)\leq \eta_\Omega(w)
$$
for $w\in\Omega$. Moreover, the second inequality is sharp.
\end{corollary}
\begin{proof}
The first inequality directly follows from \cite[Theorem~6.4]{Min16} and Theorem \ref{p3eq2}, 
while the second one always holds. It is easy to see that the equality in the second inequality holds when 
$\Omega$ is a twice punctured complex plane.
\end{proof}


Boundary of a hyperbolic domain $\Omega$ is said to be {\em uniformly perfect} if there exists a constant 
$b>0$ such that $\lambda_\Omega(w)\geq b/\delta(w)$ for all $w\in\Omega$. However, from \eqref{p3eq3} 
and the domain monotonicity property of the hyperbolic density we have the reverse inequality 
$\lambda_\Omega(w)\leq 2/\delta(w)$. This leads to the second corollary as follows.

\begin{corollary}
Let boundary of the domain $\Omega\subset\mathbb{C}$ be  uniformly perfect. Then
the hyperbolic density $\lambda_\Omega$ and the density $\overline{\eta}_\Omega$ are 
bi-Lipschitz equivalent in $\Omega$. 
\end{corollary}


Observe that Corollary \ref{p3cor1} establishes a bi-Lipschitz equivalence of the 
Hurwitz density $\eta_\Omega$ and the density $\overline{\eta}_\Omega$ on hyperbolic domains. 
However, in our next theorem we provide a sharper bound whenever domain is hyperbolic with connected boundary. 

As an application to the Riemann mapping theorem, it is easy to see that simply connected proper domains of $\mathbb{C}$ have connected boundaries. However, there exist a non-simply connected hyperbolic domain with connected boundary. For instance the exterior of the unit disk $\Omega=\{w\in\mathbb{C}:|w|>1\}$ is a non-simply connected hyperbolic domain whose boundary $\partial\Omega=\partial\mathbb{D}=\{w\in\mathbb{C}:|w|=1\}$ is a connected set.

\begin{theorem}\label{p3thm4.1}
Let $\Omega$ be a hyperbolic domain with connected boundary, then for every $w\in\Omega$ we have
$$
\overline{\eta}_\Omega(w)\leq\eta_{\Omega}(w)\leq \cfrac{K}{4}\,\overline{\eta}_\Omega(w),
$$
where $K=1/(2\lambda_{\mathbb{C}\setminus\{0,1\}}(-1))\approx 4.3859$.
\end{theorem}
\begin{proof}
Let $w\in\Omega$. Choose a point $w_1\in\partial\Omega$ such that $|w-w_1|=\inf\{|w-w^*|:w^*\in\partial\Omega\}$. Consider the circle $\Gamma=\{z\in\mathbb{C}:|z-w_1|=|w-w_1|\}$.
Since boundary of $\Omega$ is connected, there exists a point $w_2\in\partial\Omega\cap\Gamma$. Since the function $h(w)=(w-w_1)/(w_2-w_1)$ from $\mathbb{C}\setminus\{w_1,w_2\}$ onto $\mathbb{C}\setminus\{0,1\}$ is a M\"obius transformation, by Lemma~\ref{p3lem2} it is an infinitesimal isometry of the $\overline{\eta}$ density. As a consequence, we have
$$
\overline{\eta}_\Omega(w)=|h'(w)|\overline{\eta}_{h(\Omega)}(h(w))
=\cfrac{1}{|w_2-w_1|}\,\overline{\eta}_{h(\Omega)}\left(\cfrac{w-w_1}{w_2-w_1}\right).
$$
Since on hyperbolic domains, the Hurwitz density exceeds the hyperbolic density, we obtain
\begin{align*}
\overline{\eta}_\Omega(w)
& \geq\cfrac{1}{|w_2-w_1|}\,\overline{\lambda}_{h(\Omega)}\left(\cfrac{w-w_1}{w_2-w_1}\right)\\
& \geq\cfrac{1}{|w_2-w_1|}\cfrac{1}{\left(2\left|\cfrac{w-w_1}{w_2-w_1}\right|\left|\log\left|\cfrac{w-w_1}{w_2-w_1}\right|\right|+2K\left|\cfrac{w-w_1}{w_2-w_1}\right|\right)},
\end{align*}
where the last inequality follows from \cite[Theorem 2]{Hem79}.
Since $w_2\in\Gamma$, it follows that
$$
\overline{\eta}_\Omega(w)\geq\cfrac{1}{2K|w-w_1|}.
$$
The desired result follows from the inequality $1/8|w-w_1|\geq\eta_{\Omega}$. 
\end{proof}


\medskip
\noindent 
{\bf Acknowledgement.}
The research work of Arstu is 
supported by CSIR-UGC (Grant No: 21/06/2015(i)EU-V) and of S.K. Sahoo is
supported by MATRICS-SERB (Grant No: MTR/2019/001630).

\end{document}